\documentclass[a4paper,12pt]{article}

\usepackage[left=2cm,right=2cm, top=2cm,bottom=3cm,bindingoffset=0cm]{geometry}

\usepackage{verbatim}
\usepackage{amsmath}
\usepackage{amsthm}
\usepackage{amssymb}
\usepackage{delarray}
\usepackage{cite}
\usepackage{hyperref}
\usepackage{mathrsfs}
\usepackage{tikz}
\usetikzlibrary{patterns}
\usepackage{caption}
\DeclareCaptionLabelSeparator{dot}{. }
\newcommand{\de}{\delta}
\newcommand{\la}{\lambda}
\newcommand{\om}{\omega}

\newcommand{\vv}{\varphi}

\theoremstyle{plain}

\numberwithin{equation}{section}

\newtheorem{thm}{Theorem}[section]
\newtheorem{lem}[thm]{Lemma}
\newtheorem{prop}[thm]{Proposition}
\newtheorem{cor}[thm]{Corollary}

\theoremstyle{definition}

\newtheorem{alg}[thm]{Algorithm}
\newtheorem{ip}[thm]{Inverse Problem}

\theoremstyle{remark}

 \allowdisplaybreaks

\begin{document}

\begin{center}
{\Large\bf Uniform stability for the Hochstadt-Lieberman problem}
\\[0.5cm]
{\bf Natalia P. Bondarenko}
\end{center}

\vspace{0.5cm}

{\bf Abstract.} In this paper, we prove the uniform stability of the Hochstadt-Lieberman problem, which consists in the recovery of the Sturm-Liouville potential on a half-interval from the spectrum and the known potential on the other half-interval. For this purpose, we reduce the half-inverse problem to the complete one for the unknown potential. Our method relies on the uniform stability for the direct and inverse Sturm-Liouville problems, for recovering sine-type functions from their zeros, and the uniform boundedness of Riesz bases of sines and cosines.

\medskip


\medskip


\section{Introduction} \label{sec:intr}

Consider the Sturm-Liouville eigenvalue problem
\begin{gather} \label{eqv}
-y'' + q(x) y = \rho^2 y, \quad x \in (0, 2\pi), \\ \label{bc}
y'(0) - h y(0) = 0, \quad y'(2\pi) + H y(2\pi) = 0,
\end{gather}
where $q$ is a real-valued function of $L_2(0,2\pi)$, $h, H \in \mathbb R$, and $\rho^2$ is the spectral parameter, $\arg \rho \in \left[ -\tfrac{\pi}{2}, \tfrac{\pi}{2}\right)$. Denote by $\{ \rho_n^2 \}_{n \ge 1}$ the eigenvalues of the boundary value problem \eqref{eqv}--\eqref{bc}, numbered in the increasing order: $\rho_n^2 < \rho_{n+1}^2$, $n \ge 1$.

The Hochstadt-Lieberman inverse problem, which is also called the half-inverse problem and the inverse Sturm-Liouville problem with mixed data, is formulated as follows:

\begin{ip} \label{ip:HL}
Given the mixed data $\mathfrak S := \bigl( q(x)_{|x \in (\pi, 2\pi)}; \, H; \, \{\rho_n\}_{n \ge 1} \bigr)$, find $q(x)$ on $(0,\pi)$ and $h$.
\end{ip}

The original paper by Hochstadt and Lieberman \cite{HL78} is concerned only with the uniqueness for solution of the half-inverse problem. 
Methods for finding the solution, basing on transformation operators, were proposed by Sakhnovich \cite{Sakh01} and by Hryniv and Mykytyuk \cite{HM04}. Another reconstruction method, based on interpolation of entire functions, was developed in parallel by Buterin \cite{But09, But11} and by Martinyuk and Pivovarchik \cite{MP10}.
Local stability of the Hochstadt-Lieberman inverse problem, to the best of the author's knowledge, for the first time has been proved in \cite{BB17} for the self-adjoint case and in \cite{Bond20} for the non-self-adjoint case. 

This paper aims to prove the uniform stability of Inverse Problem~\ref{ip:HL}. Unconditional uniform stability for the classical inverse Sturm-Liouville problems by two spectra and by the spectral function has been shown by Savchuk and Shkalikov \cite{SS10, SS13} (see also their previous works \cite{SS06, SS08}) and by Hryniv \cite{Hryn11}. For non-self-adjoint Sturm-Liouville operators, uniform stability of inverse problems was studied in \cite{Bond24}. Recently, a significant progress was achieved in investigation of the uniform stability for inverse spectral problems for various classes of nonlocal operators (see, e.g., \cite{But21, BD22, Kuz23}).
In contrast to previous studies, while analyzing the uniform stability for the Hochstadt-Lieberman problem, we have to consider perturbations not only for the spectrum $\{ \rho_n^2 \}_{n \ge 1}$ but also for the given part of the potential $q(x)_{|x \in (\pi, 2\pi)}$, which makes the problem more challenging.

Proceed to formulating of our main result.
The eigenvalues possess the well-known asymptotics (see, e.g., \cite{FY01}):
\begin{equation} \label{asympt}
\rho_n = \frac{n-1}{2} + \frac{\om}{\pi n} + \frac{\varkappa_n}{n}, \quad \{\varkappa_n\} \in l_2, \quad  n \ge 1,
\end{equation}
where
\begin{equation} \label{defom}
\om := h + H + \frac{1}{2} \int_0^{2\pi} q(x) \, dx.
\end{equation}
Knowing $\{ \rho_n \}_{n \ge 1}$, one can determine $\om$ and $\{ \varkappa_n \}_{n \ge 1}$ as follows:
\begin{equation} \label{findom}
\om = \lim_{n \to \infty} \left( \rho_n - \frac{n-1}{2} \right) \pi n, \quad \varkappa_n = n \left( \rho_n - \frac{n-1}{2} \right) - \frac{\om}{\pi}.
\end{equation}

Consider two collections of mixed data $\mathfrak S^{(i)} = \bigl( q^{(i)}(x)_{|x \in (\pi, 2\pi)}; \, H^{(i)}; \, \{\rho_n^{(i)}\}_{n \ge 1} \bigr)$, $i = 1, 2$. We use the upper indices $^{(i)}$ for denoting objects related to the corresponding data $\mathfrak S^{(i)}$ ($i = 1, 2$).
Define the distance
$$
d(\mathfrak S^{(1)}, \mathfrak S^{(2)}) := \| q^{(1)} - q^{(2)} \|_{L_2(\pi, 2\pi)} + |H^{(1)} - H^{(2)}| + |\om^{(1)} - \om^{(2)}| + \sqrt{\sum_{n \ge 1} |\varkappa_n^{(1)} - \varkappa_n^{(2)}|^2}.
$$

For $Q > 0$, introduce the ball
$$
P_Q := \bigl\{ (q, h, H) \in L_2(0,2\pi) \times \mathbb R \times \mathbb R \colon \| q \|_{L_2(0,2\pi)} + |h| + |H| \le Q \bigr\}.
$$

Below, we denote by $C(Q)$ ($C(Q_+)$, $C(Q_-)$, etc.) various positive constants which depend only on $Q$ ($Q_+$, $Q_-$, etc.).

The uniform stability of Inverse Problem~\ref{ip:HL} is presented in the following theorem.

\begin{thm} \label{thm:main}
For any $(q^{(1)}, h^{(1)}, H^{(1)})$ and $(q^{(2)}, h^{(2)}, H^{(2)})$ in $P_Q$ and the corresponding mixed data $\mathfrak S^{(1)}$ and $\mathfrak S^{(2)}$, the following uniform estimate holds:
\begin{equation} \label{uni}
\| q^{(1)} - q^{(2)} \|_{L_2(0,\pi)} + |h^{(1)} - h^{(2)}| \le C(Q)\,d(\mathfrak S^{(1)}, \mathfrak S^{(2)}).
\end{equation}
Thus, the inverse mapping $\mathfrak S \mapsto \bigl( q(x)_{|x \in (0,\pi)}; \, h \bigr)$ is Lipschitz continuous on the set of mixed data corresponding to $(q, h, H) \in P_Q$.
\end{thm}

The proof of Theorem~\ref{thm:main} is based on using the uniform stability of the direct Sturm-Liouville problem \cite{SS10}, the uniform boundedness of Riesz bases of sines and cosines \cite{Hryn10}, the uniform stability for recovering sine-type functions from their zeros \cite{But22} and for solving the inverse Sturm-Liouville problem by the Cauchy data \cite{Bond24}.

\section{Preliminaries}

Denote by $\varphi(x, \rho)$ and $\psi(x, \rho)$ the solutions of equation \eqref{eqv} satisfying the initial conditions
$$
\vv(0,\rho) = 1, \quad \vv'(0, \rho) = h, \quad
\psi(2\pi, \rho) = 1,  \quad \psi'(2\pi, \rho) = -H.
$$
Then, the eigenvalues of the problem \eqref{eqv}--\eqref{bc} coincide with the squares of the zeros of the characteristic function
\begin{equation} \label{defDelta}
\Delta(\rho) := \vv'(\pi, \rho) \psi(\pi, \rho) - \vv(\pi, \rho) \psi'(\pi, \rho).
\end{equation}

The functions from \eqref{defDelta} admit the following representations, which are obtained by using transformation operators (see \cite{Mar11}):
\begin{align} \label{intvv}
\vv(\pi, \rho) & = \cos \rho \pi + \om_- \frac{\sin \rho \pi}{\rho} + \frac{1}{\rho} \int_0^{\pi} \mathscr K_0(t) \sin \rho t \, dt, \\ \label{intvvp}
\vv'(\pi, \rho) & = -\rho \sin \rho \pi + \om_- \cos \rho \pi + \int_0^{\pi} \mathscr K(t) \cos \rho t \, dt, \\
\psi(\pi, \rho) & = \cos \rho \pi + \om_+ \frac{\sin \rho \pi}{\rho} + \frac{1}{\rho} \int_0^{\pi} \mathscr N_0(t) \sin \rho t \, dt, \\ \label{intpsip}
\psi'(\pi, \rho) & = \rho \sin \rho \pi - \om_+ \cos \rho \pi + \int_0^{\pi} \mathscr N(t) \cos \rho t \, dt, \\ \label{intDelta}
\Delta(\rho) & = -\rho \sin 2 \rho \pi + \om \cos 2 \rho \pi + \int_0^{\pi} \mathscr M(t) \cos 2 \rho t \, dt,
\end{align}
where $\om$ is defined by \eqref{defom},
\begin{equation} \label{om+}
\om_- := h + \frac{1}{2} \int_0^{\pi} q(x) \, dx, \quad
\om_+ := H + \frac{1}{2} \int_{\pi}^{2\pi} q(x) \, dx, \quad
\om_- + \om_+ =\om,
\end{equation}
and $\mathscr K_0$, $\mathscr K$, $\mathscr N$, $\mathscr N_0$, $\mathscr M$ are real-valued functions of $L_2(0,\pi)$. 

For $Q_+ > 0$, consider the ball
$$
R_{Q_+} := \{ (q, H) \in L_2(\pi, 2\pi) \times \mathbb R \colon \| q \|_{L_2(\pi, 2\pi)} + |H| \le Q_+ \}.
$$
Relying on the construction of transformation operators \cite{Mar11} and using the technique of \cite{But21}, one can easily show that
\begin{align} \label{estN}
& \quad \| \mathscr N_0 \|_{L_2(0,\pi)} + \| \mathscr N \|_{L_2(0, \pi)} \le C(Q_+) \quad \text{for} \:\: (q, H) \in R_{Q_+}, \\ \nonumber
\| \mathscr N_0^{(1)} - & \mathscr N_0^{(2)} \|_{L_2(0, \pi)} + 
\| \mathscr N^{(1)} - \mathscr N^{(2)} \|_{L_2(0, \pi)} \\ \label{difN}
& \le C(Q_+) \bigl( \| q^{(1)} - q^{(2)} \|_{L_2(\pi, 2\pi)} + |H^{(1)} - H^{(2)}|\bigr) \quad \text{for} \:\: (q^{(i)}, H^{(i)}) \in R_{Q_+}, \:\: i = 1, 2.
\end{align}

The squares of the zeros of the functions $\psi(\pi, \rho)$ and $\psi'(\pi, \rho)$ are the eigenvalues of the boundary value problems for equation \eqref{eqv} on $(\pi, 2\pi)$ with the boundary conditions
$$
y(\pi) = 0, \:\: y'(2\pi) + H y(2\pi) = 0 \quad \text{and} \quad y'(\pi) = 0, \:\: y'(2\pi) + H y(2\pi) = 0,
$$
respectively. 
Denote these eigenvalues by $\{ \mu_n^2 \}_{n \ge 1}$ and $\{ \nu_n^2 \}_{n \ge 1}$, respectively, in the increasing order. They possess the standard asymptotics (see, e.g., \cite{FY01}):
\begin{align} \label{asymptmu}
\mu_n & = n - \frac{1}{2} + \frac{\om_+}{\pi n} + \frac{\xi_n}{n}, \quad \{\xi_n \} \in l_2, \\ \nonumber
\nu_n & = n - 1 + \frac{\om_+}{\pi n} + \frac{\tau_n}{n}, \quad \{ \tau_n \} \in l_2.
\end{align}

Without loss of generality, we will assume that 
\begin{equation} \label{bound0}
\mu_n, \, \nu_n \ge \frac{1}{2}, \quad n \ge 1.
\end{equation}
This condition can be achieved by a shift $q(x) := q(x) + s$, $\rho^2 := \rho^2 + s$, where $s$ is a large positive constant. More precisely, for every $Q_+ > 0$, one can choose the same shift $s = s(Q_+)$ for all $(q, H) \in R_{Q_+}$ to get the lower bound \eqref{bound0}.

We will need the following lemma, which is analogous to the direct part of Theorem 2.12 from the paper \cite{SS10} by Savchuk and Shkalikov.

\begin{lem} \label{lem:stabd}
(i) For every $Q_+ > 0$, there exist $\delta = \delta(Q_+) > 0$ and $\Omega = \Omega(Q_+)$ such that, for any $(q, H) \in R_{Q_+}$ satisfying \eqref{bound0}, the following estimates hold:
\begin{align*} 
& \mu_{n+1} - \mu_n \ge \de, \:\: n \ge 1, \quad |\om_+| + \| \{ \xi_n \} \|_{l_2} \le \Omega, \\ \nonumber
& \nu_{n+1} - \nu_n \ge \de, \:\: n \ge 1, \quad |\om_+| + \| \{ \tau_n \} \|_{l_2} \le \Omega.
\end{align*}
Furthermore, for any two parameter pairs $(q^{(i)}, H^{(i)}) \in R_{Q_+}$, $i = 1, 2$, both satisfying \eqref{bound0}, the corresponding spectra fulfill the uniform stability estimates
\begin{equation}  \label{difxin}
|\om_+^{(1)} - \om_+^{(2)}| + \bigl\| \{ \xi_n^{(1)} - \xi_n^{(2)} \} \bigr\|_{l_2} + \bigl\| \{ \tau_n^{(1)} - \tau_n^{(2)} \} \bigr\|_{l_2} \le
C(Q_+) \bigl( \| q^{(1)} - q^{(2)} \|_{L_2(\pi, 2\pi)} + |H^{(1)} - H^{(2)}|\bigr).
\end{equation}
Thus, the mapping $\bigl( q(x)_{|x \in (\pi, 2\pi)}, H \bigr) \mapsto \left( \om_+, \{ \xi_n \}_{n \ge 1}, \{ \tau_n \}_{n \ge 1} \right)$ is Lipschitz continuous on $R_{Q_+}$.

(ii)  For every $Q > 0$, there exists $\Omega = \Omega(Q)$ such that, for any $(q, h, H) \in P_Q$, there holds $\{ \rho_n \}_{n \ge 1} \in B_{\Omega}$, where
$$
B_{\Omega} := \bigl\{ \{ \rho_n \}_{n \ge 1} \colon \eqref{asympt} \:\: \text{is fulfilled and} \:\: |\om| + \| \{ \varkappa_n \}\|_{l_2} \le \Omega \bigr\}.
$$
\end{lem}

Recall that a sequence $V = \{ v_n \}_{n \ge 1}$ in a Hilbert space $\mathscr H$ is called \textit{a Riesz basis} if $V$ is equivalent to some orthonormal basis $\{ e_n \}_{n \ge 1}$ in $\mathscr H$, that is, there exists a bounded and boundedly invertible linear operator $T$ in $\mathscr H$ such that $v_n = T e_n$, $n \ge 1$. A family $\mathscr V$ of Riesz bases is called \textit{uniformly bounded} if the norms $\| T \|$ and $\| T^{-1} \|$ are uniformly bounded for all $V \in \mathscr V$. The uniform boundedness of Riesz bases of exponentials, sines, and cosines has been studied in \cite{Hryn10}. In particular, we will rely on the following result.

\begin{prop}[\cite{Hryn10}] \label{prop:Rb}
Suppose that $\Omega > 0$, $\de \in (0,1)$, and $s \in [0,1]$.
Denote by $\mathscr L_{\Omega,\de,s}$ the family of the increasing sequences of non-negative numbers $\{\la_n \}_{n \ge 1}$ such that $\la_{n + 1} - \la_n \ge \de$ for all $n \ge 1$ and $\| \{ \la_n - n + s \} \|_{l_2} \le \Omega$.
Then the cosine sequences $\{ \cos \la_n x \}_{n \ge 1}$ for $\{ \la_n \}_{n \ge 1} \in \mathscr L_{\Omega,\de,s}$ with a fixed $s \in [1/2, 1]$ and the sine sequences $\{ \sin \la_n x \}_{n \ge 1}$ for $\{ \la_n \}_{n \ge 1} \in \mathscr L_{\Omega,\de,s}$ with a fixed $s \in [0,1/2]$ form uniformly bounded families of Riesz bases in $L_2(0, \pi)$.
\end{prop}

Lemma~\ref{lem:stabd} and Proposition~\ref{prop:Rb} imply the following corollary.

\begin{cor} \label{cor:Rb}
Under the hypothesis of part (i) of Lemma~\ref{lem:stabd}, the Riesz bases $\{ \sin \mu_n t \}_{n \ge 1}$, and $\{ \cos \nu_n t \}_{n \ge 1}$ are uniformly bounded in $L_2(0,\pi)$. The corresponding operators $T$ satisfy the estimates $\| T \|, \, \| T^{-1} \| \le C(Q_+)$.
\end{cor}

The function $\Delta(\rho)$ can be recovered from its zeros as an infinite product (see, e.g., \cite{FY01}):
\begin{equation} \label{prodDelta}
\Delta(\rho) = 2 \pi (\rho_1^2 - \rho^2) \prod_{n = 1}^{\infty} \frac{\rho_{n+1}^2 - \rho^2}{(n/2)^2}.
\end{equation}

We need the following lemma on the uniform boundedness and stability for recovering the function $\mathscr M$ in \eqref{intDelta} from the zeros $\{ \rho_n \}_{n \ge 1}$ of $\Delta(\rho)$.

\begin{lem} \label{lem:difM}
Suppose that $\{ \rho_n^{(i)} \}_{n \ge 1} \in B_\Omega$ for some $\Omega > 0$ and $i = 1, 2$. Then
\begin{align} \label{estM}
\| \mathscr M^{(i)} \|_{L_2(0,\pi)} & \le C(\Omega), \quad i = 1, 2, \\ \label{difM}
\| \mathscr M^{(1)} - \mathscr M^{(2)} \|_{L_2(0,\pi)} & \le C(\Omega) \left( |\om^{(1)} - \om^{(2)}| + \sqrt{\sum_{n \ge 1} |\varkappa_n^{(1)} - \varkappa_n^{(2)}|^2} \right).
\end{align}
\end{lem}

For the case $\om^{(i)} = 0$, the uniform estimates \eqref{estM} and \eqref{difM} have been proved in \cite{But22}. The estimates of Lemma~\ref{lem:difM} in the general case $\om^{(i)} \ne 0$ are obtained by applying a shift and using technical calculations, which were presented, e.g., in \cite{But21} for another type of boundary conditions.

Finally, consider the problem of reconstruction of $q(x)_{|x \in (0,\pi)}$ and $h$ from the functions $\vv(\pi, \rho)$ and $\vv'(\pi,\rho)$ represented by \eqref{intvv} and \eqref{intvvp}, respectively. We call the collection $( \mathscr K, \mathscr K_0, \om_- )$ \textit{the Cauchy data} of the Sturm-Liouville problem for equation \eqref{eqv} on $(0,\pi)$ with the boundary conditions
$$
y'(0) - h y(0) = 0, \quad y'(\pi) = 0.
$$

Using the Cauchy data, one can find the Weyl function $\dfrac{\vv'(\pi, \rho)}{\vv(\pi,\rho)}$ by formulas \eqref{intvv} and \eqref{intvvp} and then uniquely reconstruct $q(x)_{|x \in (0,\pi)}$ and $h$ by standard methods (see, e.g., \cite{FY01}). The uniform stability of recovering $q$ and $h$ from the Cauchy data has been proved in \cite{Bond24}:

\begin{prop}[\cite{Bond24}] \label{prop:Cauchy}
Suppose that $|h^{(i)}| + \| q^{(i)} \|_{L_2(0,\pi)} \le Q_-$ for a constant $Q_- > 0$ and $i = 1, 2$. Then
$$
\| q^{(1)} - q^{(2)} \|_{L_2(0, \pi)} + |h^{(1)} - h^{(2)}| \le C(Q_-) \bigl( \| \mathscr K^{(1)} - \mathscr K^{(2)} \|_{L_2(0, \pi)} + \| \mathscr K_0^{(1)} - \mathscr K_0^{(2)} \|_{L_2(0,\pi)} + |\om_-^{(1)} - \om_-^{(2)}|\bigr).
$$
\end{prop}

\section{Proofs}

In this section, we prove Theorem~\ref{thm:main}.

Let us substitute \eqref{intvv} and \eqref{intvvp} into \eqref{defDelta}. Putting $\rho = \mu_n$ and $\rho = \nu_n$, we obtain the relations
\begin{equation} \label{relK}
\int_0^{\pi} \mathscr K_0(t) \sin \mu_n t \, dt = k_{0,n}, \quad \int_0^{\pi} \mathscr K(t) \cos \nu_n t \, dt = k_n, \quad n \ge 1,
\end{equation}
where
\begin{align} \label{defk0n}
k_{0, n} & := -\mu_n \left( \cos \mu_n \pi + \om_- \frac{\sin \mu_n \pi}{\mu_n} + \frac{\Delta(\mu_n)}{\psi'(\pi, \mu_n)}\right),  \\ \label{defkn}
k_n & := \nu_n \sin \nu_n \pi - \om_- \cos \nu_n \pi + \frac{\Delta(\nu_n)}{\psi(\pi, \nu_n)}.
\end{align}

Our proof of Theorem~\ref{thm:main} relies on the following theoretical algorithm for solving Inverse Problem~\ref{ip:HL}:

\begin{alg} \label{alg:sol}
Suppose that the mixed data $\mathfrak S = \bigl( q(x)_{|x \in (\pi, 2\pi)}; H; \{ \rho_n \}_{n \ge 1} \bigr)$ are given. We have to find $q(x)_{|x \in (0,\pi)}$ and $h$.
\begin{enumerate}
\item Using $\{ \rho_n \}_{n \ge 1}$, construct $\om$ and $\Delta(\rho)$ by formulas \eqref{findom} and \eqref{prodDelta}, respectively.
\item Using $q(x)_{|x \in (\pi, 2\pi)}$ and $H$, find $\om_+$ by \eqref{om+}, $\om_- := \om - \om_+$, $\{ \mu_n \}_{n \ge 1}$, $\{ \nu_n \}_{n \ge 1}$, $\psi(\pi,\nu_n)$ and $\psi'(\pi, \mu_n)$ for $n \ge 1$.
\item Find $\{ k_{0,n} \}_{n \ge 1}$ and $\{ k_n \}_{n \ge 1}$ by \eqref{defk0n} and \eqref{defkn}, respectively.
\item Determine $\mathscr K_0(t)$ and $\mathscr K(t)$ satisfying \eqref{relK}.
\item Recover $q(x)_{|x \in (0,\pi)}$ and $h$ from the Cauchy data $( \mathscr K, \mathscr K_0, \om_- )$.
\end{enumerate}
\end{alg}

Algorithm~\ref{alg:sol} is analogous to the interpolation method of \cite{But09, But11, MP10}.

\begin{lem} \label{lem:estkn}
For any $\mathfrak S^{(i)} \in R_{Q_+} \times B_{\Omega}$ satisfying \eqref{bound0}, $i = 1, 2$, the following uniform estimates hold:
\begin{align*} 
\| \{ k_{0,n}^{(i)}\} \|_{l_2} + \| \{ k_n^{(i)} \} \|_{l_2} & \le C(Q_+, \Omega), \quad i = 1, 2, \\ 
\| \{ k_{0,n}^{(1)} - k_{0,n}^{(2)} \} \|_{l_2} + \| \{ k_{n}^{(1)} - k_{n}^{(2)} \} \|_{l_2} & \le C(Q_+, \Omega) d(\mathfrak S^{(1)}, \mathfrak S^{(2)}).
\end{align*}
\end{lem}

\begin{proof}
Let us prove the desired estimates for $\{ k_{0,n}^{(i)} \}$. The proof for $\{ k_n^{(i)} \}$ is analogous.

Consider two collections $\mathfrak S^{(i)}$ ($i = 1,2$) satisfying the hypothesis of this lemma.
Denote
\begin{equation} \label{defCZ}
C_n^{(i)} := \bigl(\mu_n^{(i)} \cos \mu_n^{(i)} \pi + \om_-^{(i)} \sin \mu_n^{(i)} \pi\bigr) \tfrac{d}{dx}\psi^{(i)}(\pi, \mu_n^{(i)}) + \mu_n^{(i)} \Delta^{(i)}(\mu_n^{(i)}), \quad Z_n^{(i)} := \tfrac{d}{dx} \psi^{(i)}(\pi, \mu_n^{(i)}).
\end{equation}
Thus, $k_{0,n}^{(i)} = -\dfrac{C_n^{(i)}}{Z_n^{(i)}}$ and so
\begin{equation} \label{difCZ}
|k_{0,n}^{(1)} - k_{0,n}^{(2)}| \le \frac{|C_n^{(1)} - C_n^{(2)}|}{|Z_n^{(1)}|} + \frac{|C_n^{(2)}| |Z_n^{(1)} - Z_n^{(2)}|}{|Z_n^{(1)}||Z_n^{(2)}|}.
\end{equation}

Let us estimate the numerators and the denominators in \eqref{difCZ}. 
It follows from \eqref{asymptmu} and Lemma~\ref{lem:stabd} that
\begin{gather} \label{estmu}
|\mu_n^{(i)}| \le C(Q_+) n, \quad |\mu_n^{(1)} - \mu_n^{(2)}| \le C(Q_+)\frac{\| q^{(1)} - q^{(2)}\|_{L_2(\pi, 2\pi)} + |H^{(1)} - H^{(2)}|}{n}, \\
|\cos \mu_n^{(i)} \pi| \le \frac{C(Q_+)}{n}, \quad |\cos \mu_n^{(1)}\pi - \cos \mu_n^{(2)} \pi| \le C(Q_+) |\mu_n^{(1)} - \mu_n^{(2)}|, \\
|\sin \mu_n^{(i)} \pi| \le C(Q_+), \quad |\sin \mu_n^{(1)} \pi - \sin \mu_n^{(2)} \pi| \le C(Q_+) \frac{|\mu_n^{(1)} - \mu_n^{(2)}|}{n}
\end{gather}
for $n \ge 1$, $i = 1, 2$. 
Using \eqref{estN}, \eqref{difN}, and \eqref{estmu}, we estimate the integral term in \eqref{intpsip} and obtain
\begin{align} \label{estintN}
& \left| \int_0^{\pi} \mathscr N^{(i)}(t) \cos \mu_n^{(i)} t \, dt \right| \le C(Q_+)\| \mathscr N^{(i)} \|_{L_2(0,\pi)} \le C(Q_+), \\ \nonumber
& \left| \int_0^{\pi} \mathscr N^{(1)}(t) \cos \mu_n^{(1)} t \, dt - \int_0^{\pi} \mathscr N^{(2)}(t) \cos \mu_n^{(2)} t \, dt\right| \\ \nonumber & \le \left| \int_0^{\pi} (\mathscr N^{(1)}(t) - \mathscr N^{(2)}(t)) \cos \mu_n^{(1)} t \, dt \right| + \left| \int_0^{\pi} \mathscr N^{(2)}(t) (\cos \mu_n^{(1)} t - \cos \mu_n^{(2)} t) \, dt \right| \\ \nonumber & \le C(Q_+) \| \mathscr N^{(1)} - \mathscr N^{(2)} \|_{L_2(0,\pi)} + C(Q_+)\| \mathscr N^{(2)} \|_{L_2(0,\pi)} |\mu_n^{(1)} - \mu_n^{(2)}| \\ \label{difintN}
& \le C(Q_+) (\| q^{(1)} - q^{(2)}\|_{L_2(\pi, 2\pi)} + |H^{(1)} - H^{(2)}|).
\end{align}

Using \eqref{intpsip}, \eqref{defCZ}, and the estimates \eqref{estmu}--\eqref{difintN}, we get
\begin{gather} \label{estZ}
0 < c(Q_+) n \le |Z_n^{(i)}| \le C(Q_+) n, \quad i = 1, 2, \\ \label{difZ}
|Z_n^{(1)} - Z_n^{(2)}| \le C(Q_+) (\| q^{(1)} - q^{(2)}\|_{L_2(\pi, 2\pi)} + |H^{(1)} - H^{(2)}|), \quad n \ge 1.
\end{gather}

Next, substituting \eqref{intpsip} and \eqref{intDelta} into the formula \eqref{defCZ} for $C_n^{(i)}$, we derive
\begin{align} \nonumber
\frac{C_n^{(i)}}{\mu_n^{(i)}} = & -\mu_n^{(i)} \cos \mu_n^{(i)} \pi \sin \mu_n^{(i)} \pi + \om_-^{(i)} \cos^2 \mu_n^{(i)} \pi - \om_+^{(i)} \sin^2 \mu_n^{(i)}\pi \\ \label{smCn} & + \left( \cos \mu_n^{(i)}\pi + \om_-^{(i)} \frac{\sin \mu_n^{(i)} \pi}{\mu_n^{(i)}}\right) \int_0^{\pi} \mathscr N^{(i)}(t) \cos\mu_n^{(i)}t \, dt + \int_0^{\pi} \mathscr M^{(i)}(t) \cos 2 \mu_n^{(i)}t \, dt.
\end{align}
Note that
\begin{equation} \label{estom-}
|\om_-^{(i)}| \le |\om^{(i)}| + |\om_+^{(i)}| \le \Omega + Q_+, \quad |\om_-^{(1)} - \om_-^{(2)}| \le |\om^{(1)} - \om^{(2)}| + |\om_+^{(1)} - \om_+^{(2)}|
\end{equation}
Using \eqref{smCn} together with the asymptotics \eqref{asymptmu}, the estimates \eqref{estN}, \eqref{estmu}--\eqref{difintN}, \eqref{estom-}, and Lemma~\ref{lem:difM}, we obtain
\begin{align} \label{estC}
|C_n^{(i)}| & \le C(Q_+, \Omega) \bigl( n |\xi_n^{(i)}| + n |\hat{\mathscr M}_n^{(i)}| + 1 \bigr), \\ \label{difC}
|C_n^{(1)} - C_n^{(2)}| & \le C(Q_+, \Omega) \bigl( n |\xi_n^{(1)} - \xi_n^{(2)}| + n |\hat{\mathscr M}_n^{(1)} - \hat{\mathscr M}_n^{(2)} | + d(\mathfrak S^{(1)}, \mathfrak S^{(2)}) \bigr)
\end{align}
for $n \ge 1$, $i = 1, 2$, where $\hat {\mathscr M}_n^{(i)} := \int_0^{\pi} \mathscr M^{(i)}(t) \cos (2n-1)t \, dt$.

Combining \eqref{difCZ}, \eqref{estZ}, \eqref{difZ}, \eqref{estC}, and \eqref{difC}, we get
\begin{align*}
|k_{0,n}^{(i)}| & \le C(Q_+, \Omega) \bigl(|\xi_n^{(i)}| + |\hat{\mathscr M}_n^{(i)}| + n^{-1} \bigr), \\
|k_{0,n}^{(1)} - k_{0,n}^{(2)}| & \le C(Q_+, \Omega) \left( |\xi_n^{(1)} - \xi_n^{(2)}| + |\hat{\mathscr M}_n^{(1)} - \hat{\mathscr M}_n^{(2)} | + \frac{d(\mathfrak S^{(1)}, \mathfrak S^{(2)})}{n}\right).
\end{align*}
By summation, using Bessel's inequalities
$$
\sqrt{\sum_{n \ge 1} |\hat{\mathscr M}_n^{(i)}|^2} \le C \| \mathscr M^{(i)} \|_{L_2(0,\pi)}, \quad 
\sqrt{\sum_{n \ge 1} |\hat{\mathscr M}_n^{(1)} - \hat{\mathscr M}_n^{(2)} |^2} \le C \| \mathscr M^{(1)} - \mathscr M^{(2)} \|_{L_2(0, \pi)}
$$
together with Lemmas~\ref{lem:stabd} and~\ref{lem:difM}, we arrive at the estimates
$$
\sqrt{\sum_{n \ge 1} |k_{0,n}^{(i)}|^2} \le C(Q_+, \Omega), \quad 
\sqrt{\sum_{n \ge 1} |k_{0,n}^{(1)} - k_{0,n}^{(2)}|^2} \le C(Q_+, \Omega) d(\mathfrak S^{(1)}, \mathfrak S^{(2)}),
$$
which conclude the proof.
\end{proof}

\begin{lem} \label{lem:estCauchy}
For any $\mathfrak S^{(i)} \in R_{Q_+} \times B_{\Omega}$ $(i = 1, 2)$ satisfying \eqref{bound0}, the quantities $\mathscr K^{(i)}$, $\mathscr K_0^{(i)}$, and $\om_-^{(i)}$ obtained by Algorithm~\ref{alg:sol} fulfill the estimate
\begin{equation} \label{estCauchy}
\| \mathscr K^{(1)} - \mathscr K^{(2)} \|_{L_2(0, \pi)} + \| \mathscr K_0^{(1)} - \mathscr K_0^{(2)} \|_{L_2(0,\pi)} + |\om_-^{(1)} - \om_-^{(2)}| \le C(Q_+, \Omega) d(\mathfrak S^{(1)}, \mathfrak S^{(2)}).
\end{equation}
Thus, the mapping $\mathfrak S \mapsto ( \mathscr K, \mathscr K_0, \om_- )$ is Lipschitz continuous on $R_{Q_+} \times B_{\Omega}$.
\end{lem}

\begin{proof}
Let $\mathfrak S^{(i)}$ ($i = 1, 2$) fulfill the hypothesis of the lemma.
Due to Corollary~\ref{cor:Rb}, the sequences $\{ v_n^{(i)} \}_{n \ge 1}$, $v_n^{(i)} := \sin \mu_n^{(i)} t$, $i = 1, 2$, are uniformly bounded Riesz bases in $L_2(0, \pi)$. Let $\{ v_n^{*(i)} \}_{n \ge 1}$ be the corresponding biorthonormal bases: $(v_n^{(i)}, v_k^{*(i)}) = \de_{n,k}$, where $\de_{n,k}$ is the Kronecker delta. Due to the definition of Riesz basis, $v_n^{(i)} = T^{(i)} e_n$ and $v_n^{*(i)} = (T^{(i)*})^{-1} e_n$, where $\{ e_n \}_{n \ge 1}$ is an orthonormal basis in $L_2(0,\pi)$ and $T^{(i)}$, $i = 1,2$, are the corresponding bounded linear operators.

Using \eqref{relK}, we obtain
$$
\mathscr K_0^{(i)}(t) = \sum_{n \ge 1} k_{0,n}^{(i)} v_n^{*(i)}(t), \quad i = 1,2.
$$
Hence
\begin{align} \nonumber
\mathscr K_0^{(1)} - \mathscr K_0^{(2)} & = \sum_{n \ge 1} \bigl( k_{0,n}^{(1)} - k_{0,n}^{(2)} \bigr) v_n^{*(1)} + \sum_{n \ge 1} k_{0,n}^{(2)} \bigl( v_n^{*(1)} - v_n^{*(2)}\bigr) \\ \label{reldifK0}
& = \bigl( T^{(1)*} \bigr)^{-1} \sum_{n \ge 1} \bigl( k_{0,n}^{(1)} - k_{0,n}^{(2)}\bigr) e_n + \bigl[ \bigl( T^{(1)*}\bigr)^{-1} - \bigl( T^{(2)*}\bigr)^{-1}\bigr] \sum_{n \ge 1} k_{0,n}^{(2)} e_n
\end{align}

By Corollary~\ref{cor:Rb}, we have
\begin{equation} \label{estT1}
\| (T^{(i)})^{-1} \| \le C(Q_+), \quad \| ( T^{(1)})^{-1} - ( T^{(2)})^{-1} \| \le C(Q_+) \| T^{(1)} - T^{(2)} \|.
\end{equation}
Note that
$$
(T^{(1)} - T^{(2)}) e_n(t) = \cos \mu_n^{(1)} t - \cos \mu_n^{(2)} t, \quad n \ge 1.
$$
Therefore, it follows from \eqref{estmu} that
\begin{equation} \label{estT2}
\| T^{(1)} - T^{(2)} \| \le C(Q_+) \bigl( \| q^{(1)} - q^{(2)} \|_{L_2(\pi, 2\pi)} + |H^{(1)} - H^{(2)}|\bigr).
\end{equation}
The estimates analogous to \eqref{estT1} and \eqref{estT2} are valid for the adjoint operators $T^{(i)*}$, $i = 1, 2$.

Combining \eqref{reldifK0}--\eqref{estT2} and the estimates of Lemma~\ref{lem:estkn}, we get
$$
\| \mathscr K_0^{(1)} - \mathscr K_0^{(2)} \|_{L_2(0,\pi)} \le C(Q_+, \Omega) d(\mathfrak S^{(1)}, \mathfrak S^{(2)}).
$$
Similarly, we obtain the same estimate for $\mathscr K^{(1)} - \mathscr K^{(2)}$. 
The desired estimate for $|\om_-^{(1)} - \om_-^{(2)}|$ follows from \eqref{estom-} and \eqref{difxin}. As a result, we arrive at \eqref{estCauchy}.
\end{proof}

\begin{proof}[Proof of Theorem~\ref{thm:main}]
Suppose that $(q^{(i)}, h^{(i)}, H^{(i)}) \in P_Q$ ($i = 1, 2$). Using part (ii) of Lemma~\ref{lem:stabd}, we conclude that the mixed data $\mathfrak S^{(i)}$ belong to $R_{Q_+} \times B_{\Omega}$, where $Q_+ = Q$ and $\Omega = \Omega(Q)$. Then, Lemma~\ref{lem:estCauchy} implies the estimate \eqref{estCauchy} with $C = C(Q)$. Applying Proposition~\ref{prop:Cauchy} with $Q_- = Q$ completes the proof.
\end{proof}

\section{Discussion}

Note that Lemma~\ref{lem:estCauchy} establishes the unconditional uniform stability for reconstruction of the triple $( \mathscr K, \mathscr K_0, \om_- )$ by using any data $\mathfrak S \in R_{Q_+} \times B_{\Omega}$. In other words, the constant $C$ in the estimate \eqref{estCauchy} depends only on the bounds $Q_+$ and $\Omega$ for the given data $\mathfrak S$.
However, not every $\mathfrak S \in R_{Q_+} \times B_{\Omega}$ are the mixed data for some problem of form \eqref{eqv}--\eqref{bc}. Therefore,
the triple $( \mathscr K, \mathscr K_0, \om_- )$, which is considered in Lemma~\ref{lem:estCauchy}, is not necessarily the Cauchy data corresponding to some potential $q(x)_{|x \in (0,\pi)}$ and $h$.

The known existence theorems for solution of the half-inverse problem either have a local nature (see, e.g., \cite{Sakh01, HM04, BB17}) or rely on reduction to a complete inverse spectral problem on the half-interval $(0,\pi)$ and impose some a posteriori requirements on the spectral data of that auxiliary problem (see, e.g., \cite{HM04, MP10, Bond20}). This does not allow us to find suitable a priori requirements on $\mathfrak S$ for the unconditional uniform stability of Inverse Problem~\ref{ip:HL}. Therefore, in Theorem~\ref{thm:main}, we impose the a priori bound $(q, h, H) \in P_Q$ on the parameters of the problem \eqref{eqv}--\eqref{bc} and so obtain the conditional uniform stability. Anyway, it is worth pointing out that this a priori bound is crucial only for the last step of the proof, that is, for the recovery of $q(x)_{|x \in (0,\pi)}$ and $h$ from the Cauchy data. All the previous steps of Algorithm~\ref{alg:sol} possess the unconditional uniform stability as summarized in Lemma~\ref{lem:estCauchy}.

It is worth mentioning that, for the special case $q^{(1)} = q^{(2)}$ a.e. on $(\pi, 2\pi)$ and $H^{(1)} = H^{(2)}$, the estimate \eqref{uni} can be obtained by developing the ideas of Horvath and Kiss \cite{HK10}. Indeed, Theorem~1.10 in \cite{HK10} contains the estimate analogous to \eqref{uni} (for the case of the Dirichlet boundary conditions) with the constant $C$, which depends on (i) $\| q \|_{L_1(0,\pi)}$, (ii) the bound of the Riesz basis $\{ \cos 2 \rho_n t \}_{n \ge 1}$. Using the uniform stability of the direct problem by Savchuk and Shkalikov \cite{SS10} and Proposition~\ref{prop:Rb} by Hryniv, one can easily show that the Riesz basis $\{ \cos 2 \rho_n t \}_{n \ge 1}$ is uniformly bounded and so $C = C(Q)$ for $(q, h, H) \in P_Q$. However, Horvath and Kiss \cite{HK10} have not considered the uniform boundedness of cosine bases and confined themselves by local stability results.

We also mention that our method does not work for the non-self-adjoint case, since the mapping $\bigl(q(x)_{|x \in (\pi, 2\pi)}, H\bigr) \mapsto \bigl( \{ \mu_n \}_{n \ge 1}, \{ \nu_n \}_{n \ge 1} \bigr)$ may violate Lipschitz continuity for multiple eigenvalues.

\medskip

{\bf Acknowledgement.} The author is grateful to Prof. Sergey A. Buterin for valuable discussions.

\medskip

{\bf Funding.} This work was supported by Grant 24-71-10003 of the Russian Science Foundation, https://rscf.ru/en/project/24-71-10003/.

\noindent Natalia Pavlovna Bondarenko \\

\noindent 1. Department of Mechanics and Mathematics, Saratov State University, 
Astrakhanskaya 83, Saratov 410012, Russia, \\

\noindent 2. Department of Applied Mathematics and Physics, Samara National Research University, \\
Moskovskoye Shosse 34, Samara 443086, Russia, \\

\noindent 3. S.M. Nikolskii Mathematical Institute, Peoples' Friendship University of Russia (RUDN University), 6 Miklukho-Maklaya Street, Moscow, 117198, Russia, \\

\noindent 4. Moscow Center of Fundamental and Applied Mathematics, Lomonosov Moscow State University, Moscow 119991, Russia.\\

\noindent e-mail: {\it bondarenkonp@sgu.ru}


\medskip

\begin{thebibliography}{00}


\bibitem{HL78}
Hochstadt, H.; Lieberman, B. An inverse Sturm-Liouville problem with mixed given data, SIAM J. Appl. Math. 34 (1978), 676--680.

\bibitem{Sakh01}
Sakhnovich, L. Half-inverse problems on the finite interval, Inverse Problems 17 (2001), 527--532.

\bibitem{HM04}
Hryniv, R.O.; Mykytyuk Ya.V. Half-inverse spectral problems for Sturm-Liouville operators with singular potentials, Inverse Problems 20 (2004), 1423--1444.

\bibitem{But09}
Buterin, S.A. On a constructive solution of the incomplete inverse Sturm-Liouville problem, Mathematika. Mekhanika, Saratov State University, Saratov 11 (2009), 8--12 [in Russian].

\bibitem{But11}
Buterin, S.A. On half inverse problem for differential pencils with the spectral parameter in boundary conditions, Tamkang J. Math. 42 (2011), 355--364

\bibitem{MP10}
Martinyuk, O.; Pivovarchik, V. On the Hochstadt-Lieberman theorem, Inverse Problems 26 (2010), 035011 (6pp).

\bibitem{BB17}
Bondarenko, N.; Buterin, S. On a local solvability and stability of the inverse transmission eigenvalue problem,  Inverse Problems 33 (2017), 115010.

\bibitem{Bond20}
Bondarenko, N.P. Solvability and stability of the inverse Sturm-Liouville problem with analytical functions in the boundary condition, Math. Meth. Appl. Sci. 43 (2020), no. 11, 7009-7021. 


\bibitem{SS10}
Savchuk, A.M.; Shkalikov, A.A. Inverse problems for Sturm-Liouville operators with potentials in Sobolev spaces: Uniform stability, Funct. Anal. Appl. 44 (2010), no. 4, 270--285.

\bibitem{SS13}
Savchuk, A.M.; Shkalikov, A.A. Uniform stability of the inverse Sturm-Liouville problem with respect to the spectral function in the scale of Sobolev spaces, Proc. Steklov Inst. Math. 283 (2013), 181--196.

\bibitem{SS06}
Savchuk, A.M.; Shkalikov, A.A. On the eigenvalues of the Sturm-Liouville operator with potentials from Sobolev spaces, Math. Notes 80 (2006), no.~6, 814--832.

\bibitem{SS08}
Savchuk, A.M.; Shkalikov, A.A. On the properties of maps connected with inverse Sturm-Liouville problems, Proc. Steklov Inst. Math. 260 (2008), 218--237.

\bibitem{Hryn11}
Hryniv, R.O. Analyticity and uniform stability in the inverse singular Sturm-Liouville spectral problem, Inverse Problems 7 (2011), 065011.

\bibitem{Bond24}
Bondarenko, N.P. Uniform stability of the inverse problem for the non-self-adjoint Sturm-Liouville operator, arXiv:2409.16175.

\bibitem{But21}
Buterin, S. Uniform full stability of recovering convolutional perturbation of the Sturm-Liouville operator from the spectrum,
J. Diff. Eqns. 282 (2021), 67--103.

\bibitem{BD22}
Buterin, S.; Djuric, N. Inverse problems for Dirac operators with constant delay: uniqueness, characterization, uniform stability, Lobachevskii J. Math. 43 (2022), no. 6, 1492--1501.

\bibitem{Kuz23}
Kuznetsova, M. Uniform stability of recovering Sturm-Liouville-type operators with frozen argument, Results Math. 78 (2023), no.~5, Article ID 169.


\bibitem{FY01}
Freiling, G.; Yurko, V. Inverse Sturm-Liouville Problems and Their Applications, Nova Science Publishers, Huntington, NY (2001).


\bibitem{Hryn10}
Hryniv, R.O. Uniformly bounded families of Riesz bases of exponentials, sines, and cosines, Math. Notes 87 (2010), no.~4m 510--520.

\bibitem{But22}
Buterin, S.A. On the uniform stability of recovering sine-type functions with asymptotically separated zeros, Math. Notes 111 (2022), no.~3, 343--355.

\bibitem{Mar11}
Marchenko, V.A. Sturm-Liouville Operators and Applications. Revised edition, AMS, Providence, 2011.

\bibitem{HK10}
Horv\'ath, M.; Kiss, M. Stability of direct and inverse eigenvalue problems for Schr\"odinger operators on finite intervals, Inter. Math. Research Notices 2010 (2010), no.~11, 2022--2063.

\end{thebibliography}
\end{document}